\documentclass[12pt, a4paper, leqno]{amsart}

\usepackage{amsmath}
\usepackage{amsfonts}
\usepackage{amssymb}
\usepackage{color}
\usepackage[T1]{fontenc} 
\usepackage{url}

\numberwithin{equation}{section}

\newtheorem{theorem}{Theorem}[section]
\newtheorem{definition}[theorem]{Definition}
\newtheorem{lemma}[theorem]{Lemma}
\newtheorem{proposition}[theorem]{Proposition}
\newtheorem{corollary}[theorem]{Corollary}

\theoremstyle{remark}
\newtheorem{remark}[theorem]{Remark}
\newtheorem{example}[theorem]{Example}

\theoremstyle{plain}

\providecommand{\loc}{{\ensuremath{\mathrm{loc}}}}

\newcommand{\W}{\mathcal{W}^{\alpha}_{\alpha_1,\alpha_2}}
\newcommand{\w}{\textbf{\textit{w}}}

\newcommand{\px}{{p(\cdot)}}
\newcommand{\qx}{{q(\cdot)}}

\newcommand{\qzx}{{q_0(\cdot)}}
\newcommand{\qumx}{{q_1(\cdot)}}

\newcommand{\A}{A^{\textbf{\textit{w}}}_{\px,\qx}}
\newcommand{\B}{B^{\textbf{\textit{w}}}_{\px,\qx}}
\newcommand{\F}{F^{\textbf{\textit{w}}}_{\px,\qx}}

\newcommand{\Nz}{\ensuremath{\mathbb{N}_0}}
\newcommand{\R}{\mathbb{R}}
\newcommand{\N}{\mathbb{N}}
\newcommand{\Z}{\mathbb{Z}}

\newcommand{\cS}{\mathcal{S}}

\newcommand{\Rn}{{\mathbb{R}^n}}
\newcommand{\Zn}{{\mathbb{Z}^n}}

\DeclareMathOperator{\supp}{supp}

\def\esssup{\operatornamewithlimits{ess\,sup}}
\def\essinf{\operatornamewithlimits{ess\,inf}}
\newcommand{\PPlog}{\mathcal{P}^{\log}(\Rn)}
\newcommand{\PP}{\mathcal{P}(\Rn)}

\begin{document}

\title{On $2$-microlocal spaces with all exponents variable}

\author[A. Almeida]{Alexandre Almeida$^{*}$}
\address{Center for R\&D in Mathematics and Applications, Department of Mathematics, University of Aveiro, 3810-193 Aveiro, Portugal}
\email{jaralmeida@ua.pt}

\author[A. Caetano]{Ant\'{o}nio Caetano}
\address{Center for R\&D in Mathematics and Applications, Department of Mathematics, University of Aveiro, 3810-193 Aveiro, Portugal}
\email{acaetano@ua.pt}

\thanks{$^*$ Corresponding author.}
\thanks{This work was partially supported by Portuguese funds through CIDMA (Center for Research and Development in Mathematics and Applications) and FCT (Foundation for Science and Technology) within project UID/MAT/04106/2013.}
\thanks{\copyright 2016. Licensed under the CC-BY-NC-ND 4.0 license http://creativecommons.org/licenses/by-nc-nd/4.0/}
\thanks{Formal publication: http://dx.doi.org/10.1016/j.na.2016.01.016}

\date{\today}



\begin{abstract}
In this paper we study various key properties for $2$-microlocal Besov and Triebel--Lizorkin spaces with all exponents variable, including the lifting property, embeddings and Fourier multipliers. We also clarify and improve some statements recently published.
\end{abstract}

\maketitle


\section{Introduction}\label{sec:intro}

Function spaces with variable integrability already appeared in the work of
Orlicz \cite{Orl31}, although the modern development started with the
paper \cite{KR91} of Kov\'{a}{\v c}ik and R\'{a}kosn\'{\i}k. Corresponding
PDE with non-standard growth have been studied approximately since the same time.
For an overview we refer to the monographs \cite{C-UF13,DHHR11} and the survey \cite{HarHLN10}.
Apart from interesting theoretical investigations, the motivation to study
such function spaces comes from applications to fluid dynamics
\cite{Ruz00}, image processing \cite{CheLR06,HarHLT13,Tiir14},
PDE and the calculus of variations, see for example \cite{AceM01,Fan07,LZhang13}.

Some ten years ago, a further step was taken by Almeida and Samko \cite{AlmS06} and
Gurka, Harjulehto and Nekvinda \cite{GHN} by introducing variable exponent Bessel potential spaces $\mathcal{L}^{s,\px}$ (with constant $s$). As in the classical case, these spaces coincide with the Lebesgue/Sobolev spaces for integer $s$.
Later Xu \cite{Xu08a} considered Besov $B^{s}_{\px,q}$ and Triebel--Lizorkin
$F^{s}_{\px,q}$ spaces with variable $p$, but fixed $q$ and $s$.

In a different context, Leopold \cite{Leo89a, Leo91} considered Besov type spaces with the smoothness
index determined by certain symbols of hypoelliptic pseudo-differential operators. Special choices of such symbols lead to spaces $B^{s(\cdot)}_{p,p}$ of variable smoothness. More general function spaces with variable smoothness $B^{s(\cdot)}_{p,q}$ and $F^{s(\cdot)}_{p,q}$ were explicitly studied by
Besov \cite{Bes03}, including characterizations by differences.

More recently all the above mentioned spaces were integrated into larger scales similarly with the full classical Besov and Triebel--Lizorkin scales with constant exponents. However, such extension requires all
the indices to be variable. Such three-index generalization was done by Diening, H\"ast\"o and Roudenko
\cite{DieHR09} for Triebel--Lizorkin spaces $F^{s(\cdot)}_{\px,\qx}$, and by Almeida and H\"ast\"o \cite{AlmH10} for Besov spaces
$B^{s(\cdot)}_{\px,\qx}$. This full extension led to immediate gains, for example with the study of traces where the integrability and smoothness indices interact, see \cite{DieHR09,AlmH14}, and it also provided an important unification. Indeed, when $s\in[0,\infty)$ and $p\in\PPlog$ is bounded away from $1$ and $\infty$ then $F^s_{\px,2}=\mathcal{L}^{s,\px}$ (\cite{DieHR09}) are Bessel potential spaces, which in turn are Sobolev spaces if $s$ is integer (\cite{AlmS06}). On the other hand, the variable Besov scale above includes variable order H\"older-Zygmund spaces as special cases (cf. \cite[Theorem~7.2]{AlmH10}).

It happens that the smoothness parameter can be generalized in different directions. In the so-called $2$-microlocal spaces $B^\w_{p,q}$ and $F^\w_{p,q}$ the smoothness is measured by a certain weight sequence $\w=(w_j)_{j\in\Nz}$, which is rich enough in order to frame spaces with variable smoothness and spaces with generalized smoothness (see \cite{FarL06}). The $2$-microlocal spaces already appeared in the works of Peetre \cite{Pee75} and Bony \cite{Bo84}. Later Jaffard and Meyer \cite{Jaf91}, \cite{JafMey96}, and Lévy Véhel and Seuret \cite{LevSeu03} have also used such spaces in connection with the study of regularity properties of functions. Function spaces with constant integrability defined by more general microlocal weights were also studied by Andersson \cite{And97}, Besov \cite{Bes03}, Moritoh and Yamada \cite{MYam04} and Kempka \cite{Kem08}.

The generalization mixing up variable integrability and $2$-microlocal weights was done by Kempka \cite{Kem09} providing a unification for many function spaces studied so far. However, in the case of Besov spaces the fine index $q$ was still kept fixed.

In this paper we deal with the general Besov and Triebel-Lizorkin scales $B^w_{\px,\qx}$ and $F^w_{\px,\qx}$ on $\Rn$ with all exponents variable. After some necessary background material (Section~\ref{sec:prelim}), we discuss in Section~\ref{sec:PeetreMaxFunc} the characterization in terms of Peetre maximal functions (Theorem~\ref{thm:Peetre}) and, as a consequence, the independence of the spaces from the admissible system taken (Corollary~\ref{cor:independent}). Although such statements have already been presented by Kempka and Vyb\'{\i}ral in \cite{KemV12}, their proofs contain some unclear points, see the details in the discussions after Theorem~\ref{thm:Peetre} and Corollary~\ref{cor:means} below.

In the remaining sections we prove some key properties for both scales $B^w_{\px,\qx}$ and $F^w_{\px,\qx}$: the lifting property in Section~\ref{sec:lifting}; embeddings in Section~\ref{sec:basic-embed}; finally, Fourier multipliers in Section~\ref{sec:Fmultipliers}. For other key properties, like atomic and molecular representations and Sobolev type embeddings, we refer to our paper \cite{AlmC15b}.

We notice that recently in \cite{LSUYY13} a very general framework was proposed for studying function spaces and proving similar properties for the related spaces. Although the framework suggested over there is very general in some aspects, it does not include Besov and Triebel-Lizorkin spaces with variable $q$. This fact is very relevant, since the mixed sequences spaces behind do not share some fundamental properties as in the constant exponent situation.


\section{Preliminaries}\label{sec:prelim}

As usual, we denote by $\mathbb{R}^{n}$ the $n$-dimensional real
Euclidean space, $ \N$ the collection of all natural numbers and
$\N_{0}= \N\cup \{0\}$. By $\Zn$ we denote the lattice of all points in $\Rn$ with integer components. If $r$ is a real number then $r_+:=\max\{r,0\}$.
We write $B(x,r)$ for the open ball in
$\mathbb{R}^{n}$ centered at $x\in \mathbb{R}^{n}$ with radius $r>0$.
We use $c$ as a generic positive constant, i.e.\ a constant whose
value may change with each appearance. The expression $f
\lesssim g$ means that $f\leq c\,g$ for some independent constant
$c$, and $f\approx g$ means $f \lesssim g \lesssim f$.

The notation $X\hookrightarrow Y$ stands for continuous embeddings
from $X$ into $Y$, where $X$ and $Y$ are quasi-normed spaces. If
$E\subset {\mathbb{R}^{n}}$ is a  measurable set, then $|E|$ stands
for its (Lebesgue) measure and $\chi_{E}$ denotes its characteristic
function. By $\supp f$ we denote the support of the function $f$.

The set $\cS$ denotes the usual Schwartz class of infinitely differentiable
rapidly decreasing complex-valued functions and $\cS'$
denotes the dual space of tempered distributions. The Fourier
transform of a tempered distribution $f$ is denoted by $\hat f$ while its inverse transform is denoted by $f^\vee$.

\subsection{Variable exponents}

By $\PP$ we denote the set of all measurable functions $p:\Rn
\rightarrow (0,\infty]$ (called \textit{variable exponents}) which
are essentially bounded away from zero.  For $E\subset \Rn$
and $p\in \PP$ we denote $p_E^+ =\esssup_E p(x)$ and
$p_E^-=\essinf_E p(x)$. For simplicity we use the abbreviations $p^+=p_\Rn^+$ and
$p^-=p_\Rn^-$.

The \emph{variable exponent Lebesgue space} $L_\px=L_{\px}(\Rn)$ is the
class of all (complex or extended real-valued) measurable functions $f$ (on $\Rn$) such that
\[
\varrho_{\px}(f/\lambda):=\int_\Rn \phi_{p(x)}\left(\frac{|f(x)|}{\lambda}\right)\, dx
\]
is finite for some $\lambda>0$, where
\[
\phi_p(t) :=
\begin{cases}
t^p & \text{ if } p\in (0,\infty), \\
0 & \text{ if } p=\infty \text{ and } t\in [0,1], \\
\infty & \text{ if } p=\infty \text{ and } t\in(1,\infty]. \\
\end{cases}
\]

It is known that $\varrho_\px$ defines a semimodular (on the vector space consisting of all measurable functions on $\Rn$ which are finite a.e.), and that $L_{\px}$ becomes a quasi-Banach space with respect to the quasi-norm
\begin{align*}
    \| f|L_\px\| &:= \inf \left\{ \lambda>0 : \varrho_{\px}\left(f/\lambda\right) \leq 1\right\}.
\end{align*}
This functional defines a norm when $p^-\geq 1$.  If $p(x)\equiv p$ is constant, then $L_{\px}=L_{p}$ is the classical Lebesgue space.

It is worth noting that $L_{\px}$ has the lattice property and that the assertions $f\in L_\px$ and $\|f\,|\, L_\px\|<\infty$ are equivalent for any (complex or extended real-valued) measurable function $f$ (assuming the usual convention $\inf \varnothing=\infty$). The fundamental properties of the spaces $L_\px$, at least in the case $p^-\geq 1$, can be found in \cite{KR91} and in the recent monographs \cite{DHHR11}, \cite{C-UF13}. The definition above of $L_{\px}$ using the semimodular $\varrho_\px$ is taken from \cite{DHHR11}.

For any $p\in\PP$ we have
\[
\| f\,|\,L_\px\|^r = \left\| |f|^r | L_{\frac{\px}{r}}\right\|\,, \ \ \ \ r\in(0,\infty),
\]
and
\[
\| f+g\,|\,L_\px\| \leq \max\left\{1,2^{\frac{1}{p^-}-1}\right\} \left( \| f\,|\,L_\px\| + \| g\,|\,L_\px\|\right).
\]

An useful property (that we shall call the \textit{unit ball
property}) is that $\rho_\px(f) \leq 1$
if and only if $\| f\,|L_\px\| \leq 1$ (\cite[Lemma~3.2.4]{DHHR11}). An interesting variant of this is the following estimate
\begin{equation}\label{Lp-norm-mod}
\min \left\{ \varrho_{\px} (f)^{\frac{1}{p^-}}, \varrho_{\px} (f)^{\frac{1}{p^+}} \right\} \le \|f\,|\, L_\px\| \le \max \left\{ \varrho_{\px} (f)^{\frac{1}{p^-}}, \varrho_{\px} (f)^{\frac{1}{p^+}} \right\}
\end{equation}
for $p\in\PP$ with $p^-<\infty$, and $\varrho_{\px} (f)>0$ or $p^+<\infty$. It is proved in \cite[Lemma~3.2.5]{DHHR11} for the case $p^-\ge 1$, but it is not hard to check that this property remains valid in the case $p^- < 1$. This property is clear for constant exponents due to the
obvious relation between the quasi-norm and the semimodular in that case.

For variable exponents, H\"older's inequality holds in the form
\[
\| f\,g\,|L_1\| \leq 2\,\| f\,|L_\px\|  \| g\,|L_{p'(\cdot)}\|
\]
for $p\in\PP$ with $p^-\ge 1$, where $p'$ denotes the conjugate exponent of $p$ defined pointwisely by $\tfrac{1}{p(x)}+\tfrac{1}{p'(x)}=1, \ \ x\in\Rn$.

From the spaces $L_\px$ we can also define \emph{variable exponent Sobolev spaces} $W^{k,\px}$ in the usual way (see \cite{DHHR11}, \cite{C-UF13} and the references therein).

In general we need to assume some regularity on the exponents in order to develop a consistent theory of variable function spaces. We recall here some classes which are nowadays standard in this setting.

We say that a continuous function $g\,:\, \Rn\to \R$ is {\em locally $\log$-H\"{o}lder
continuous}, abbreviated $g \in C^{\log}_\loc$, if there exists $c_{\log}>0$ such that
  \begin{align*}
    |g(x)-g(y)| \leq \frac{c_{\log}}{\log (e + 1/|x-y|)}
  \end{align*}
for all $x,y\in\Rn$. The function $g$ is said to be {\em globally $\log$-H\"{o}lder continuous},
abbreviated $g \in C^{\log}$, if it is locally $\log$-H\"{o}lder
continuous and there exists $g_\infty \in \R$ and $C_{\log}>0$ such that
\begin{align*}
|g(x) - g_\infty| &\leq \frac{C_{\log}}{\log(e
+ |x|)}
\end{align*}
for all $x \in \R^n$. The notation $\PPlog$ is used for those variable exponents $p\in \PP$ with $\frac1p \in C^{\log}$. We shall write $c_{\log}(g)$ when we need to use explicitly the constant involved in the local $\log$-Hölder continuity of $g$. Note that all (exponent) functions in $C^{\log}_{\loc}$ are bounded.

As regards the (quasi)norm of characteristic functions on cubes $Q$ (or balls) in $\Rn$, for exponents $p\in\PPlog$ we have
\begin{equation}\label{norm-charact-func-cubes}
\| \chi_Q\,| L_\px\| \approx |Q|^{\frac{1}{p(x)}} \ \ \ \ \textrm{if} \ \ \ |Q| \le 1 \ \ \ \textrm{and} \ \ x\in Q,
\end{equation}
and
\begin{equation*}
\| \chi_Q\,| L_\px\| \approx |Q|^{\frac{1}{p_\infty}} \ \ \ \ \textrm{if} \ \ |Q| \ge 1
\end{equation*}
(see \cite[Corollary~4.5.9]{DHHR11}),  using the shortcut $\tfrac{1}{p_\infty}$ for $\big(\tfrac{1}{p}\big)_\infty$.

\subsection{Variable exponent mixed sequence spaces}

To deal with variable exponent Besov and Triebel--Lizorkin scales we need to consider appropriate mixed sequences spaces. For $p,q\in\PP$ the \textit{mixed Lebesgue-sequence space} $L_\px(\ell_\qx)$ (\cite{DieHR09}) can be easily defined through the quasi-norm
\begin{equation} \label{def:lpq}
\|(f_\nu)_\nu\,| L_\px(\ell_\qx) \| := \big\|
\|(f_\nu(x))_\nu\,| \ell_{q(x)}\|\,| L_\px\big\|
\end{equation}
on sequences $(f_\nu)_{\nu\in\Nz}$ of complex or extended real-valued measurable functions on $\Rn$. This is a norm if $\min\{p^-,q^-\} \geq 1$.
Note that $\ell_{q(x)}$ is a standard discrete Lebesgue space (for each $x\in\Rn$), and that \eqref{def:lpq} is well defined since $q(x)$ does not depend on $\nu$ and the function $x\to \|(f_\nu(x))_\nu\,| \ell_{q(x)}\|$ is always measurable when $q\in\PP$.

The ``opposite'' case $\ell_\qx(L_\px)$ is not so easy to handle. For $p,q\in\PP$, the \textit{mixed sequence-Lebesgue space} $\ell_\qx(L_\px)$ consists of all sequences $(f_\nu)_{\nu\in\Nz}$ of (complex or extended real-valued) measurable functions (on $\Rn$) such that $\varrho_{\ell_\qx(L_\px)}\big( \tfrac{1}{\mu}(f_\nu)_\nu\big ) < \infty$ for some $\mu>0$, where
\begin{equation}\label{def:lpqmod}
\varrho_{\ell_\qx(L_\px)}\big( (f_\nu)_\nu\big ) := \sum_\nu
\inf\Big\{\lambda_\nu>0: \, \varrho_{\px}\Big(f_\nu
/\lambda_\nu^{\frac1{\qx}}\Big)\le 1 \Big\}.
\end{equation}
Note that if $q^+<\infty$ then \eqref{def:lpqmod} equals the more simple form
\begin{equation}\label{def:lpqmodsimple}
\varrho_{\ell_\qx(L_\px)}\big( (f_\nu)_\nu\big ) = \sum_\nu \Big\|
|f_\nu|^{\qx} | L_{\frac{\px}{\qx}}\Big\|.
\end{equation}
The space $\ell_\qx(L_\px)$ was introduced in \cite[Definition~3.1]{AlmH10} within the framework of the so called \emph{semimodular spaces}.  It is known (\cite{AlmH10}) that
\begin{equation*}\label{def:lpqnorm}
  \|(f_\nu)_\nu\,| \ell_\qx(L_\px)\| :=
  \inf\Big\{ \mu>0:\, \varrho_{\ell_\qx(L_\px)}\big( \tfrac1\mu  (f_\nu)_\nu\big ) \le 1\Big\}
\end{equation*}
defines a quasi-norm in $\ell_\qx(L_\px)$ for every $p,q\in\PP$ and that $\|\cdot\,| \ell_\qx(L_\px)\|$ is a
norm when $q\geq 1$ is constant and $p^-\geq 1$, or when
$\frac{1}{p(x)}+\frac{1}{q(x)} \leq 1$ for all $x\in\Rn$. More recently, it was
shown in \cite{KemV13} that it also becomes a norm if $1\leq q(x)\leq p(x)\leq \infty$. Contrarily to the situation when $q$ is constant, the expression above is not necessarily a norm when $\min\{p^-,q^-\} \geq 1$ (see \cite{KemV13} for an example showing that the triangle inequality may fail in this case).

It is worth noting that $\ell_\qx(L_\px)$ is a really iterated space when $q\in(0,\infty]$ is constant (\cite[Proposition~3.3]{AlmH10}), with
\begin{equation}\label{iterated}
\|(f_\nu)_\nu\,| \ell_q(L_\px)\| = \big\| \big(\| f_\nu\,| L_\px\|\big)_\nu \,| \ell_q\big\|.
\end{equation}
We note also that the values of $q$ have no influence on $\|(f_\nu)_\nu\,| \ell_\qx(L_\px)\|$ when we consider sequences having just one non-zero entry. In fact, as in the constant exponent case, we have
$\|(f_\nu)_\nu\,|\, \ell_\qx(L_\px)\| = \|f\,|\, L_\px\|$ whenever there exists $\nu_0\in\Nz$ such that $f_{\nu_0}=f$ and $f_\nu \equiv 0$ for all $\nu \not= \nu_0$ (cf. \cite[Example~3.4]{AlmH10}).

Simple calculations show that given any sequence $(f_\nu)_\nu$ of measurable functions, finiteness of $\|(f_\nu)_\nu\,| \ell_\qx(L_\px)\|$ implies $(f_\nu)_\nu\in\ell_\qx(L_\px)$, which in turn implies $f_\nu\in L_\px$ for each $\nu\in\Nz$. Moreover,
$$\|(f_\nu)_\nu\,| \ell_\qx(L_\px)\| \le 1 \ \ \ \text{if and only if} \ \ \ \varrho_{\ell_\qx(L_\px)}\big( (f_\nu)_\nu\big ) \le 1 \ \ \ \ \ \text{(unit ball property)}$$
(see \cite{AlmH10}).

We notice that both mixed sequence spaces $L_\px(\ell_\qx)$ and $\ell_\qx(L_\px)$ satisfy the lattice property.

The next lemma can be proved following the arguments used in the proof of Theorem~6.1 (i),(iii) in \cite{AlmH10}.
\begin{lemma}\label{lem:lpqembed}
Let $p,q,q_0,q_1\in\PP$. If $q_0\le q_1$ then we have
\begin{equation*}
L_\px(\ell_\qzx) \hookrightarrow L_\px(\ell_\qumx) \ \ \ \ \text{and} \ \ \ \ \ell_\qzx(L_\px) \hookrightarrow \ell_\qumx(L_\px).
\end{equation*}
Moreover, if $p^+,q^+<\infty$ then it also holds
\begin{equation*}
\ell_{\min\{\px,\qx\}}(L_\px)   \hookrightarrow L_\px(\ell_\qx) \hookrightarrow \ell_{\max\{\px,\qx\}}(L_\px).
\end{equation*}
\end{lemma}


We notice that the Hardy-Littlewood maximal operator is not, in general, a good tool in the spaces $L_\px(\ell_\qx)$ and $\ell_\qx(L_\px)$. Indeed, as observed in \cite{AlmH10} and \cite{DieHR09}, such operator is not bounded in these spaces when $q$ is non-constant. A way of overcoming this difficulty is to use convolution inequalities involving radially decreasing kernels,
namely the so-called $\eta$-\emph{functions} having the form
\[
\eta_{\nu,R}(x) := \frac{2^{n\nu}}{(1+2^\nu|x|)^R}\,, \ \ \ x\in\Rn,
\]
with $\nu\in \Nz$ and $R>0$. Note that $\eta_{\nu,R}\in L_1$ when
$R>n$ and that $\| \eta_{\nu,R}\,|L_1\|$ depends only on $n$ and $R$.

\begin{lemma}\label{thm:conv-eta}
Let $p,q\in\PPlog$ and $(f_\nu)_\nu$ be a sequence of non-negative measurable functions on $\Rn$.
\begin{enumerate}
\item[(i)] If $1<p^-\leq p^+ <\infty$ and $1<q^-\leq q^+ <\infty$, then for $R>n$ there holds
\[
\| (\eta_{\nu,R} * f_\nu)_\nu\,| L_\px(\ell_\qx) \| \lesssim \|
(f_\nu)_\nu \,| L_\px(\ell_\qx) \|.
\]
\item[(ii)] If $p^-\ge 1$ and $R>n+c_{\log}(1/q)$, then
\[
\| (\eta_{\nu,R} * f_\nu)_\nu\,| \ell_\qx(L_\px) \| \lesssim \|
(f_\nu)_\nu \,| \ell_\qx(L_\px) \|.
\]
\end{enumerate}
\end{lemma}

The convolution inequality in (i) above was given in \cite[Theorem~3.2]{DieHR09}, while the satement (ii) was established in \cite[Lemma~4.7]{AlmH10} and \cite[Lemma~10]{KemV12}.

\subsection{Admissible weights}

\begin{definition}\label{def:weights}
Let $\alpha,\alpha_1,\alpha_2\in\R$ with $\alpha \ge 0$ and $\alpha_1 \le \alpha_2$.
We say that a sequence of positive measurable functions $\w = (w_j)_j$ belongs to class $\W$ if
\begin{enumerate}
\item[(i)] there exists $c>0$ such that
\[
0<w_j(x) \leq c\,w_j(y)\left(1+2^j|x-y|\right)^\alpha
\]
for all $j\in\Nz$ and $x,y\in\Rn$;
\item[(ii)] there holds
\[
2^{\alpha_1}\, w_j(x) \leq w_{j+1}(x) \leq 2^{\alpha_2}\, w_j(x)
\]
for all $j\in\Nz$ and $x\in\Rn$.
\end{enumerate}
\end{definition}

A sequence according to the definition above is called an \emph{admissible weight sequence}.
When we write $\w\in\W$ without any restrictions then $\alpha \ge 0$ and $\alpha_1, \alpha_2\in\R$ (with $\alpha_1\leq \alpha_2$) are arbitrary but fixed numbers.
Some useful properties of class $\W$ may be found in \cite[Remark~2.4]{Kem08}.

\begin{example}
A fundamental example of an admissible weight sequence $\w$ is that formed by the 2-microlocal weights
\[
w_j(x)= 2^{js}(1+2^j\,d(x,U))^{s'}
\]
where $s,s'\in\R$ and $d(x,U)$ is the distance of $x\in\Rn$ from the (fixed) subset $U\subset\Rn$. In this case we have $\w\in \mathcal{W}^{|s'|}_{\min\{0,s'\},\max\{0,s'\}}$. The particular case $U=\{x_0\}$ (for a given point $x_0\in\Rn$), corresponds to the typical weights
\begin{equation}\label{2-microweights}
w_j(x)= 2^{js}(1+2^j\,|x-x_0|)^{s'}.
\end{equation}
\end{example}

\begin{example}\label{ex:var-smoothness}
If $s:\Rn \rightarrow \R$ is in class $C^{\log}_{\loc}$, then the weight sequence given by
\[
w_j(x)= 2^{js(x)}
\]
is in class $\mathcal{W}^{c_{\log}(s)}_{s^-,s^+}$. This follows from the estimate
\[
2^{js(x)} \eta_{j,2R}(x-y) \lesssim 2^{js(y)} \eta_{j,R}(x-y), \ \ \ \ \ x,y\in\Rn, \ \ \ j\in\Nz,
\]
which holds for $R\ge c_{\log}(s)$ (see \cite[Lemma~19]{KemV12} as a variant of \cite[Lemma~6.1]{DieHR09}).
\end{example}

\begin{example}\label{ex:gen-smoothness}
Let $(\sigma_j)_j$ be a sequence of nonnegative real numbers satisfying
\[
d_1\sigma_j \le \sigma_{j+1} \le d_2\,\sigma_j, \ \ \ j\in\Nz,
\]
for some $d_1,d_2>0$ independent of $j$. If we define the (constant) sequence $\w$ by
\[
w_j(x) \equiv \sigma_j, \ \ \ j\in\Nz,
\]
then we see that $\w\in \mathcal{W}^{0}_{\log_2 d_1,\log_2 d_2}$.

\end{example}

\begin{example}\label{ex:weight-function}
Let $\rho(x)$ be an admissible weight function, that means
$$
0<\rho(x) \lesssim \rho(y) (1+|x-y|^2)^{\frac{\beta}{2}}\,, \ \ \ x\in\Rn \ \ \ \ (\beta \ge 0).
$$
Taking $w_j(x)= 2^{js}\rho(x)$ ($s\in\R$) we obtain an admissible sequence belonging to class $\mathcal{W}^{\beta}_{s,s}$.
\end{example}

\subsection{$2$-microlocal spaces with variable integrability}\label{sec:microlocalspaces}



We say that a pair $(\varphi,\Phi)$ of functions in $\cS$ is \emph{admissible} if
\begin{itemize}
\item $\supp \hat{\varphi} \subset \{\xi\in\Rn: \, \tfrac{1}{2} \le |\xi| \le 2 \}\, $ and $\, |\hat{\varphi}(\xi)| \ge c>0$ when $\tfrac{3}{5} \le |\xi| \le \tfrac{5}{3}$;
\item $\supp \hat{\Phi} \subset \{\xi\in\Rn: \, |\xi| \le 2 \}\, $ and $\, |\hat{\Phi}(\xi)| \ge c>0$ when $ |\xi| \le \tfrac{5}{3}$.
\end{itemize}

Set $\varphi_j:=2^{jn}\varphi(2^j\cdot)$ for $j\in\N$ and $\varphi_0:=\Phi$. Then $\varphi_j\in\cS$ and
\[
\supp \widehat{\varphi}_j \subset \{\xi\in\Rn: \, 2^{j-1} \le |\xi| \le 2^{j+1} \} \,, \ \ \ j\in\N.
\]
Such a system $\{\varphi_j\}$ is also said admissible.

Now we are ready to recall the Fourier analytical approach to function spaces of Besov and Triebel-Lizorkin type.

\begin{definition} 
Let $\w=(w_j)_j\in\W$ and $p,q\in\PP$.
\begin{enumerate}
\item[(i)] $\B$ is the set of all $f\in\cS'$ such that
\begin{equation}\label{Bnorm}
\|f\,|\B\|:= \big\| (w_j(\varphi_j\ast f))_j\,|\, \ell_\qx(L_\px)\big\| < \infty.
\end{equation}
\item[(ii)] Restricting to $p^+,q^+<\infty$, $\F$ is the set of all $f\in\cS'$ such that
\begin{equation}\label{Fnorm}
\|f\,|\F\|:= \big\| (w_j(\varphi_j\ast f))_j\, |\, L_\px(\ell_\qx)\big\| < \infty.
\end{equation}
\end{enumerate}
\end{definition}

The sets $\B$ and $\F$ become quasi-normed spaces equipped with the quasinorms \eqref{Bnorm} and \eqref{Fnorm}, respectively. As in the constant exponent case, they agree when $p=q$, i.e., $B^\w_{\px,\px}=F^\w_{\px,\px}$.

In the sequel we shall write $\A$ for short when there is no need to distinguish between  Besov and Triebel-Lizorkin spaces.

\begin{example}[$2$-microlocal spaces]
A fundamental example of $2$-microlocal spaces (from which the terminology seem to come from) are the spaces constructed from the special weight sequence given by \eqref{2-microweights}. Such spaces have been considered by Peetre \cite{Pee75}, Bony \cite{Bo84}, Jaffard and Meyer \cite{Jaf91}, \cite{JafMey96}. The latter authors have used, in particular, the spaces $H^{s,s'}_{x_0}=B^{\w}_{2,2}$ and $C^{s,s'}_{x_0}=B^{\w}_{\infty,\infty}$ in connection with the study of regularity properties of functions using wavelet tools.
%
\end{example}

\begin{example}[variable smoothness]
If $w_j(x)= 2^{js(x)}$ with $s\in C^{\log}_{\loc}$, then $\B=B^{s(\cdot)}_{\px,\qx}$  and $\F=F^{s(\cdot)}_{\px,\qx}$ are the scales of spaces with variable smoothness and integrability introduced in \cite{AlmH10} and \cite{DieHR09}, respectively.
\end{example}

\begin{example}[generalized smoothness]
When $\w=(\sigma_j)_j$ is a sequence as in Example \ref{ex:gen-smoothness}, then $\A=A^{\sigma}_{\px,\qx}$  are spaces of generalized smoothness. For constant exponents $p,q$ a systematic study of such spaces was carried out by Farkas and Leopold \cite{FarL06} (even considering more general partitions of unity in their definitions), see also the study by Moura in \cite{Mou01b}. Spaces of generalized smoothness (and constant integrability) have been introduced by Goldman \cite{Gold76} and Kalyabin and Lizorkin \cite{KalLiz87}. We note that such type of function spaces were also considered in the context of interpolation in \cite{CobF86,Mer83}.
\end{example}

\begin{example}[weighted spaces]
If $w_j(x)= 2^{js}\rho(x)$ as in Example~\ref{ex:weight-function} then we get weighted function spaces (see \cite[Chapter~4]{EdTri96} and \cite{HTri05} for constant $p$ and $q$).
\end{example}


%

For simplicity we will omit the reference to the admissible pair $(\varphi,\Phi)$ used to define the quasi-norms \eqref{Bnorm} and \eqref{Fnorm}. As we  will see below, we shall obtain the same sets $\B$ and $\F$ for different choices of such pairs, at least when $p$ and $q$ satisfy some regularity assumptions (see Corollary~\ref{cor:independent}).

\section{Maximal functions characterization}\label{sec:PeetreMaxFunc}

Given $a>0$, $f\in \cS'$ and $(\psi_j)_j \subset \cS$, the \emph{Peetre maximal functions} are defined as
\[
\big( \psi^\ast_j f\big)_a(x):= \sup_{y\in\Rn} \frac{|\psi_j\ast f(y)|}{1+|2^j(x-y)|^{a}}\,, \ \ \ x\in\Rn, \ \ j\in\Nz.
\]

It is well known that these functions constitute an important tool in the study of properties of several classical functions spaces starting with getting equivalent quasi-norms. For Besov spaces with variable smoothness and integrability this topic was studied in \cite{Drihem12} using modified versions of the Peetre maximal functions above. The characterization via maximal functions was extended to the general setting of the spaces $A^\w_{\px,\qx}$ in papers \cite[Corollary~4.7]{Kem09} (in the $F$ case) and \cite[Theorem~6]{KemV12} (in the $B$ case).

The next statement integrates both cases and makes a critical improvement to the corresponding results from those papers (see explanations below).

\begin{theorem}\label{thm:Peetre}
Let $\w\in\W$ and $L\in\Nz$ with $L>\alpha_2$. Let also $\Psi,\psi\in\cS$ be such that
\begin{equation}\label{aux2}
|\hat{\Psi}(\xi)|>0 \ \ \ \text{on} \ \ \ \{\xi\in\Rn: |\xi|\le k\varepsilon\},
\end{equation}
\begin{equation}\label{aux3}
|\hat{\psi}(\xi)|>0 \ \ \ \text{on} \ \ \ \big\{\xi\in\Rn: \tfrac{\varepsilon}{2}\le |\xi|\le k\varepsilon\big\},
\end{equation}
for some $k\in]1,2]$ and $\varepsilon >0$, and
\begin{equation}\label{aux8}
D^{\gamma} \hat{\psi}(0)=0 \ \ \ \text{for} \ \ 0\le |\gamma| < L.
\end{equation}
Define
$$
\psi_0:= \Psi \ \ \ \text{and} \ \ \ \psi_j:= 2^{jn}\psi(2^j\cdot), \ \ j\in\N.
$$
\begin{enumerate}
\item[(i)] If $p,q\in\PPlog$ and $a>\alpha+ \tfrac{n}{p^-}+c_{\log}(1/q)$, then
\begin{equation*}
\|f\,|\B\|  \approx  \big\| (w_j(\psi_j\ast f))_j\,|\, \ell_\qx(L_\px)\big\|
 \approx  \Big\| \big(w_j\big( \psi^\ast_j f\big)_a\big)_j\in \ell_\qx(L_\px)\Big\|
\end{equation*}
for all $f\in\cS'$.
\item[(ii)] If $p,q\in\PPlog$ with $\max\{p^+,q^+\}<\infty$, and $a>\alpha+\tfrac{n}{\min\{p^-,q^-\}}$, then
\begin{equation*}
\|f\,|\F\| \approx \big\| (w_j(\psi_j\ast f))_j \,|\, L_\px(\ell_\qx)\big\| \approx \Big\| \big(w_j\big( \psi^\ast_j f\big)_a\big)_j \,|\, L_\px(\ell_\qx)\Big\|
\end{equation*}
for all $f\in\cS'$.
\end{enumerate}
\end{theorem}

Instead of \eqref{aux2}, \eqref{aux3} what usually one finds in the literature (see, e.g. \cite[(1)]{Rych99}, \cite[(2.17)]{Ull12}, \cite[(4),(5)]{Kem09}, \cite[(8),(9)]{KemV12}, \cite[(3.6)]{LSUYY13}) reads, in our notation, respectively
$$
|\hat{\Psi}(\xi)|>0 \ \ \ \text{on} \ \ \ \{\xi\in\Rn: |\xi|<2\varepsilon\},
$$
$$
|\hat{\psi}(\xi)|>0 \ \ \ \text{on} \ \ \ \big\{\xi\in\Rn: \tfrac{\varepsilon}{2}< |\xi|<2\varepsilon\big\}.
$$
That is, apart from having here strict inequalities (which is a minor detail), only the case $k=2$ is usually considered. This is not a problem in itself, but, depending on what has already been proved and how the $B$ and $F$ spaces are originally defined, one might end up claiming results which are not really proved.
The problem we are about to mention does not arise in the classical case with constant exponents because the independence of the spaces from the dyadic resolution of unity taken as in \cite[Definition~2.3.1]{Tri92} has been proved (see, e.g. \cite[Theorem~2.3.2]{Tri92}) independently of the classical counterpart of our Theorem~\ref{thm:Peetre} (and with $k=2$), and therefore we can use a particular dyadic resolution of unity which fits in the hypothesis of the latter theorem in order to expand the universe of systems $\{\varphi_j\}$ which can be used in the definition of the $B$ and $F$ spaces without changing them. A summary of what we can get in this way in the classical situation, including historical remarks, can be seen in \cite[Section~1.3]{Tri06} (see, in particular, Theorem~1.7 there for the counterpart of our Theorem~\ref{thm:Peetre} with $k=2$).

However, in \cite{Kem09} and \cite{KemV12}, in a setting of variable exponents, though a theorem like Theorem~\ref{thm:Peetre} (with $k=2$) is stated, what is actually proved is that the definition of the spaces is independent of the dyadic resolutions of unity producing systems which satisfy the hypotheses of the theorem for some $\varepsilon>0$ (and $k=2$). Since not all dyadic resolutions of unity considered in those papers produce systems with such characteristics (not even if we allow $k$ to vary in $]1,2]$), that claimed independence was not completely proved.

Our point of view for the definition of the spaces is different, using in it a system $\{\varphi_j\}$ built from a so-called admissible pair as defined in the beginning of Subsection~\ref{sec:microlocalspaces}. Even so, the problem of sticking to $k=2$ in \eqref{aux2}, \eqref{aux3} is that not all admissible pairs produce systems satisfying such conditions, and therefore, taking into account the approach followed in the proof, the independence of the space from such admissible pairs would not be guaranteed. On the other hand, it is easy to see that any admissible pair forms systems satisfying \eqref{aux2}, \eqref{aux3} with $\varepsilon = \frac{6}{5}$ and $k=\frac{25}{18}$. Since \eqref{aux8} is also trivially satisfied, the claimed independence is in fact proved. In some sense this is already implicit in the statements given in the theorem, but we would like to stress it as a separate important conclusion:


\begin{corollary}\label{cor:independent}
Let $\w\in\W$ and $p,q\in\PPlog$ (with $\max\{p^+,q^+\}<\infty$ in the $F$-case). Then the spaces $\B$ and $\F$ are independent of the admissible pair $(\varphi,\Phi)$ taken in its definition, in the sense that different such pairs produce equivalent quasi-norms in the corresponding spaces.
\end{corollary}

The approach to Besov and Triebel-Lizorkin spaces that we are taking is also common in the literature -- see, e.g. \cite{FraJ85,FraJ90} in the case of constant exponents, and \cite{DieHR09} and \cite{AlmH10} in the case of spaces with variable smoothness and integrability. The independence of the variable Triebel-Lizorkin spaces from the particular admissible pair taken was proved in \cite{DieHR09} (under convenient restrictions on the exponents) with the help of the so-called $\varphi$-transform, identifying $F^{s(\cdot)}_{\px,\qx}$ with a subspace of a suitable sequence space, following the approach taken in the constant exponent situation in \cite{FraJ90}. In \cite{AlmH10} (see also \cite[Remark~1.2]{AlmH14} for a correction) the corresponding independence result for $B^{s(\cdot)}_{\px,\qx}$ was only settled for a subfamily of admissible pairs.

As pointed out above, in the present paper we give a positive answer to the question of the definition of the general $2$-microlocal spaces $\B$ and $\F$ being independent of the particular admissible pair taken. Surely there are dyadic resolutions of unity which produce systems which satisfy the conditions of Theorem~\ref{thm:Peetre}, and so it is also clear that they give rise to the same $2$-microlocal spaces with variable exponents. We are not aware, however, of results showing that, in the variable setting, any dyadic resolution of unity as in \cite[Definition~2.3.1]{Tri92} will give the same outcome.

Notice that Theorem~\ref{thm:Peetre} not only free us to be tied to a specific admissible pair, but also ensures that more general pairs can indeed be used to define the spaces instead of the admissible ones, namely pairs $(\psi,\Psi)$ satisfying the requirements of the theorem.

From Theorem \ref{thm:Peetre} we can also derive a characterization of the spaces $\A$ in terms of the so-called \emph{local means} (cf. \cite[Corollary~4.8]{Kem09} and \cite[Corollary~1]{KemV12}), which are defined by
$$
k(t,f)(x):= \langle f(x+t\cdot),k\rangle := t^{-n} \langle f,k\big(\tfrac{\cdot-x}{t}\big)\rangle\, \ \ \ \ x\in\Rn, \ \ t>0,
$$
for a $C^\infty$ function $k$ on $\Rn$ supported in $B(0,1)$ and $f\in\cS'$.

\begin{corollary}\label{cor:means}
Let $k_0,k^0$ be $C^\infty$ functions such that
$$
\supp k_0,\, \supp k^0 \subset B(0,1) \ \ \ \ \text{and} \ \ \ \ \widehat{k_0}(0),\, \widehat{k^0}(0) \ne 0.
$$
Let also $\w\in\W$ and $N\in\Nz$ be such that $2N>\alpha_2$, and take
$k^N:=\Delta^N k^0$ (the Laplacian of order $N$ of $k^0$). If $p,q\in\PPlog$, then
\[
\|f\,|\B\| \approx \| w_0\,k_0(1,f)\,|L_\px\| + \Big\| \big(w_j\,k^N(2^{-j},f)\big)_{j\in\N}\,|\, \ell_\qx(L_\px) \Big\|
\]
and
\[
\|f\,|\F\| \approx \| w_0\,k_0(1,f)\,|L_\px\| + \Big\| \big(w_j\,k^N(2^{-j},f)\big)_{j\in\N}\,|\, L_\px(\ell_\qx) \Big\|
\]
for all $f\in\cS'$ (with the additional restriction $\max\{p^+,q^+\}<\infty$ in the $F$ case).
\end{corollary}

Theorem~\ref{thm:Peetre} can be proved following the general structure of the proof done by Rychkov \cite{Rych99} in the classical case (see also \cite{Ull12} for some correction of the argument) overcoming the difficulties caused by the consideration of the general exponents. As we can see in \cite{Kem09} and \cite{KemV12}, where the variable exponent case was treated, the proof involves various technical auxiliary results, for instance on the behavior of the Peetre maximal operators on variable mixed sequence spaces. Another key tool in this approach is a kind of discrete convolution inequality. A corresponding inequality within the framework of both variable mixed sequence spaces was pointed out in full generality in \cite{KemV12}, but the arguments used there are unclear to us.

It is not our aim to repeat here the arguments leading to the proof of Theorem~\ref{thm:Peetre} but, due to the difficulty pointed out above, we shall give here a proof of the mentioned discrete convolution inequality (see Lemma~\ref{lem:discrete-conv}), using completely different arguments. At the same time this will be an opportunity to exhibit the kind of difficulties that may appear when we are dealing  with variable mixed sequence spaces, specially with $\ell_\qx(L_\px)$, when compared with the constant exponent situation.

\begin{lemma}\label{lem:discrete-conv}
Let  $p,q\in \PP$ and $\delta>0$. For any sequence $(g_k)_k$ of nonnegative measurable functions on $\R^n$, consider
\begin{equation}\label{gnu}
G_\nu(x):= \sum_{k=0}^\infty 2^{-|k-\nu|\delta}\, g_k(x)\, , \ \ \ \ \ x\in\R^n, \ \ \ \nu\in\N_0.
\end{equation}
Then we have
\begin{equation}\label{discrete-conv-1}
\| (G_\nu)_\nu\,| L_\px(\ell_\qx) \| \lesssim \| (g_\nu)_\nu \,| L_\px(\ell_\qx) \|.
\end{equation}
and
\begin{equation}\label{discrete-conv-2}
\| (G_\nu)_\nu\,| \ell_\qx(L_\px) \| \lesssim \| (g_\nu)_\nu \,| \ell_\qx(L_\px) \|
\end{equation}
\end{lemma}

\begin{proof}
\emph{Step 1}: The inequalities above need to be shown essentially for $p,q \ge 1$. In fact, if they hold in such case, then for arbitrary exponents $p,q\in \PP$ we can always take $r\in \left(0, \min\{1,p^-,q^-\}\right)$ and proceed as follows:
\begin{eqnarray*}
\| (G_\nu)_\nu\,| \ell_\qx(L_\px) \|^r  &= & \left\| (G_\nu^r)_\nu\,| \ell_{\frac{\qx}{r}}\big(L_{\frac{\px}{r}}\big) \right\|
\leq  \left\|  \sum_{k\ge 0} 2^{-|k-\nu|\delta\,r}\,  g_k^r  \,| \ell_{\frac{\qx}{r}}\big(L_{\frac{\px}{r}}\big) \right\| \\
& \lesssim & \left\| (g_k^r)_k  \,| \ell_{\frac{\qx}{r}}\big(L_{\frac{\px}{r}}\big) \right\| = \| (g_k)_k \,| \ell_\qx(L_\px) \|.
\end{eqnarray*}
The argument works in the same way for the other inequality.

\emph{Step 2}: We prove \eqref{discrete-conv-1} for $p,q\ge 1$.\\
For any $s\in[1,\infty]$ and $\delta>0$, from Minkowski's inequality we can show that
\[
\left\| \left(\sum_{k\geq 0} 2^{-|k-\nu|\delta} a_k\right)_{\nu\geq 0}\, | \ell_s \right\| \leq \frac{2}{1-2^{-\delta}}\, \left\| (a_k)_{k\geq 0}\, | \ell_s \right\|
\]
for every sequence $(a_k)_{k\geq 0}$ of nonnegative numbers. So for each $x\in\Rn$ we get
\[
\left\| \left(G_\nu\right)_{\nu\geq 0}\, | \ell_{q(x)} \right\| \leq \frac{2}{1-2^{-\delta}}\, \left\| (g_k(x))_{k\geq 0}\, | \ell_{q(x)} \right\|.
\]
Taking the $L_\px$-norm in both sides and using the lattice property of $L_\px$, we get inequality \eqref{discrete-conv-1}.

\emph{Step 3}: We prove \eqref{discrete-conv-2} for $p,q\ge 1$ and $q^+<\infty$.\\
Suppose that  $\| (g_\nu)_\nu \,| \ell_\qx(L_\px) \| < \infty$ (otherwise there is nothing to prove). We want to show that there exists a constant $c>0$ such that
$$
\| (G_\nu)_\nu\,| \ell_\qx(L_\px) \|^t \leq c\,\| (g_\nu)_\nu \,| \ell_\qx(L_\px) \|^t =:c\,\mu
$$
(where $t>0$ is a certain fixed number that is chosen below in a convenient way). This is equivalent to show that, for $\mu>0$,
$$
\left\| \frac{(G_\nu)_\nu}{(c\,\mu)^{1/t}}\,| \ell_\qx(L_\px) \right\| \leq 1,
$$
that is (by the unit ball property),
$$
\varrho_{\ell_\qx(L_\px)}\left(\frac{(G_\nu)_\nu}{(c\,\mu)^{1/t}} \right) \leq 1,
$$
which we shall do next.
Take $t=\min\{1, (p/q)^-\}$. Using \eqref{def:lpqmodsimple} we have
\begin{eqnarray*}
\varrho_{\ell_\qx(L_\px)}\left(\frac{(G_\nu)_\nu}{(c\,\mu)^{1/t}}\right) & = & \sum_{\nu \geq 0} \left\| \left|\frac{G_\nu}{(c\,\mu)^{1/t}}\right|^{\qx} | L_{\frac{\px}{\qx}}\right\| = \sum_{\nu \geq 0} \left\| \left(\frac{|G_\nu|^t}{c\,\mu}\right)^{\qx} | L_{\frac{\px}{t\qx}}\right\|^{1/t} \\
& \leq & \sum_{\nu \geq 0} \left\| \left((c\,\mu)^{-1} \sum_{l\geq -\nu} 2^{-|l|\delta\,t} \, |g_{\nu+l}|^t\right)^{\qx} | L_{\frac{\px}{t\qx}}\right\|^{1/t}.
\end{eqnarray*}
For each $x\in\R^n$ and $\nu\in\N_0$, from Hölder's inequality (with the convention $g_j \equiv 0$ for $j<0$) we get
\[
\sum_{l\geq -\nu} 2^{-|l|\delta\,t} \, |g_{\nu+l}(x)|^t \leq \left(\sum_{l\in\Z} 2^{-|l|\frac{\delta}{2}\,t q'(x)}\right)^{\frac{1}{q'(x)}} \, \left(\sum_{l\in\Z} 2^{-|l|\frac{\delta}{2}\,t q(x)} |g_{\nu+l}(x)|^{tq(x)}\right)^{\frac{1}{q(x)}}
\]
(with the usual modification when $q'(x)=\infty$). Using this pointwise estimate, letting $c=c_1\cdot c_2$ with $c_1:= \sum_{l\in\Z} 2^{-|l|\frac{\delta}{2}\,t}$ and $c_2\geq 1$, and applying Minkowski's inequality  twice, we derive
\begin{eqnarray*}
\varrho^t_{\ell_\qx(L_\px)}\left(\frac{(G_\nu)_\nu}{(c\,\mu)^{1/t}}\right) & \leq & \left[\sum_{\nu \geq 0} \left( \sum_{l\in \Z} c_2^{-1} 2^{-|l|\frac{\delta}{2}\,t} \left\| \left(\frac{|g_{\nu+l}|^t}{\mu}\right)^{\qx} | L_{\frac{\px}{t\qx}}\right\| \right)^{1/t} \right]^t\\
& \leq & \sum_{l\in \Z} c_2^{-1} 2^{-|l|\frac{\delta}{2}\,t} \left(\sum_{\nu \geq 0} \left\| \left(\frac{|g_{\nu+l}|^t}{\mu}\right)^{\qx} | L_{\frac{\px}{t\qx}}\right\|^{1/t} \right)^t.
\end{eqnarray*}
Choosing $c_2=c_1$ as above (and consequently $\sqrt{c}= \sum_{l\in\Z} 2^{-|l|\frac{\delta}{2}\,t}$), we obtain
\[
\varrho_{\ell_\qx(L_\px)}\left(\frac{(G_\nu)_\nu}{(c\,\mu)^{1/t}}\right) \leq \sum_{k \geq 0} \left\| \left(\frac{|g_{k}|}{\mu^{1/t}}\right)^{t\qx} | L_{\frac{\px}{t\qx}}\right\|^{1/t} = \varrho_{\ell_\qx(L_\px)}\left(\frac{(g_k)_k}{\mu^{1/t}}\right) \leq 1,
\]
where the last inequality follows from the hypothesis that the corresponding quasi-norm is less than or equal to $1$ (recall the definition of $\mu$ above).

\emph{Step 4}: We prove \eqref{discrete-conv-2} for any $p,q\ge 1$ (including the case $q^+=\infty$).\\
Let us assume that $\| (g_\nu)_\nu \,| \ell_\qx(L_\px) \| < \infty$. Again by the unit ball property, \eqref{discrete-conv-2} will follow from the inequality
\begin{equation}\label{aux6}
\varrho_{\ell_\qx(L_\px)}\left(\frac{(G_\nu)_\nu}{c\,\mu}\right) \leq 1,
\end{equation}
where $\mu = \| (g_\nu)_\nu \,| \ell_\qx(L_\px) \|>0$ and $c>0$ is a constant independent of $\mu$. Let us then prove \eqref{aux6}.\\
Under the convention $g_j \equiv 0$ for $j<0$, we get
\begin{eqnarray}
\varrho_{\ell_\qx(L_\px)}\left(\frac{(G_\nu)_\nu}{c\,\mu}\right) & = & \sum_{\nu \geq 0} \inf\left\{\lambda_\nu>0: \, \varrho_{\px}\left(\frac{G_\nu}{c\,\mu\,\lambda_\nu^{1/\qx}} \right) \le 1 \right\} \nonumber\\
& = & \sum_{\nu \geq 0} \inf\left\{\lambda_\nu>0: \, \left\|\frac{G_\nu}{c\,\mu\,\lambda_\nu^{1/\qx}} \,| L_\px\right\|\le 1 \right\} \nonumber \\
& \leq & \sum_{\nu \geq 0} \inf\left\{\lambda_\nu>0: \, c^{-1} \sum_{l\in\Z} 2^{-|l|\delta} \left\|\frac{g_{\nu+l}}{\mu\,\lambda_\nu^{1/\qx}} \,| L_\px\right\|\le 1 \right\} \label{inf}
\end{eqnarray}
after using Minkowski's inequality in the last step of this chain. Let us define
\[
I_{\nu,l}:= \inf\left\{\lambda>0: \, c^{-1} c(\delta)2^{-|l|\delta/2} \left\|\frac{g_{\nu+l}}{\mu\,\lambda^{1/\qx}} \,| L_\px\right\|\le 1 \right\}\,, \ \ \ \nu\in\Nz, \ \ \ l\in\Z,
\]
with $c(\delta):= \sum_{l\in\Z} 2^{-|l|\delta/2}$. We claim that, for each $\nu\in\Nz$, the sum $\sum_{l\in\Z} I_{\nu,l}$ is not smaller than the infimum in \eqref{inf}.
To prove this we can obviously assume that this sum is finite. For any $\varepsilon >0$ we have
$$
c^{-1} c(\delta) 2^{-|l|\delta/2} \left\|\frac{g_{\nu+l}}{\mu\,(I_{\nu,l} + \varepsilon 2^{-|l|})^{1/\qx}} \,| L_\px\right\|\le 1,
$$
so
$$
c^{-1} c(\delta) \sum_{l\in\Z} 2^{-|l|\delta} \left\|\frac{g_{\nu+l}}{\mu\,(I_{\nu,l} + \varepsilon 2^{-|l|})^{1/\qx}} \,| L_\px\right\|\le \sum_{l\in\Z} 2^{-|l|\delta/2}.
$$
Recalling the definition of $c(\delta)$, we also obtain
$$
c^{-1} \sum_{l\in\Z} 2^{-|l|\delta} \left\|\frac{g_{\nu+l}}{\mu\,\Big(\sum_{k\in\Z} (I_{\nu,k} + \varepsilon 2^{-|k|})\Big)^{1/\qx}} \,| L_\px\right\|\le 1,
$$
and hence
$$
\inf\left\{\lambda>0: \, c^{-1} \sum_{l\in\Z} 2^{-|l|\delta} \left\|\frac{g_{\nu+l}}{\mu\,\lambda^{1/\qx}} \,| L_\px\right\|\le 1 \right\} \leq \sum_{k\in\Z} I_{\nu,k} + \varepsilon \sum_{k\in\Z} 2^{-|k|}.
$$
Since the second series on the right-hand side converges and $\varepsilon>0$ is arbitrary, we get the desired estimate. Now using it in \eqref{inf} and making a convenient change of variables (choosing the constant $c\ge c(\delta)$ and noting that $q\ge 1$), we have
\begin{eqnarray*}
\varrho_{\ell_\qx(L_\px)}\left(\frac{(G_\nu)_\nu}{c\,\mu}\right) & \leq & \sum_{\nu \geq 0} \sum_{k\in\Z} \inf\left\{\lambda>0: \, c^{-1} c(\delta) 2^{-|k|\delta/2} \left\|\frac{g_{\nu+k}}{\mu\,\lambda^{1/\qx}} \,| L_\px\right\|\le 1 \right\}\\
& \leq & \sum_{\nu \geq 0} \sum_{k\in\Z} c^{-1} c(\delta) 2^{-|k|\delta/2} \inf\left\{\sigma>0: \, \left\|\frac{g_{\nu+k}}{\mu\,\sigma^{1/\qx}} \,| L_\px\right\|\le 1 \right\}\\
& = &  \sum_{k\in\Z} c^{-1} c(\delta) 2^{-|k|\delta/2} \sum_{\nu \geq 0} \inf\left\{\sigma>0: \, \left\|\frac{g_{\nu+k}}{\mu\,\sigma^{1/\qx}} \,| L_\px\right\|\le 1 \right\}\\
& \leq &  \sum_{k\in\Z} c^{-1} c(\delta) 2^{-|k|\delta/2} \,\varrho_{\ell_\qx(L_\px)}\left(\frac{(g_j)_j}{\mu}\right)\\
& \leq & 1
\end{eqnarray*}
with the choice $c=c(\delta)^2$, taking into account our definition of $\mu$.
\end{proof}

\begin{remark}
Of course we could have omitted Step 3 in the proof above since the arguments in Step 4 work in fact for arbitrary exponents $q$. Nevertheless, giving a presentation of a separate proof for bounded exponents $q$ we want to stress that this case is more intuitive and closer to the constant exponent situation.
\end{remark}

\section{Lifting property}\label{sec:lifting}

Given $\sigma\in\R$, the lifting operator $I_\sigma$ is defined by
$$
I_\sigma f:=\Big[ \big(1+|\cdot|^2\big)^{\sigma/2} \, \hat{f}\Big]^\vee\, , \ \ \ f\in\cS'.
$$
It is well known that $I_\sigma$ is a linear one-to-one mapping of $\cS'$ onto itself and that its restriction to $\cS$ is also a one-to-one mapping of $\cS$ onto itself.

The next proposition gives a lifting property for the $2$-microlocal spaces of variable integrability.

\begin{theorem}\label{thm:lift}
Let $\w\in\W$ and $p,q\in\PPlog$ ($\max\{p^+,q^+\}<\infty$ in the $F$ case). Then $I_\sigma$ maps $\A$ isomorphically onto $A^{(-\sigma)\w}_{\px,\qx}$, where $(-\sigma)\w:=(2^{-j\sigma}w_j)_j$.\\ Moreover, $\big\|I_\sigma \cdot \,| A^{(-\sigma)\w}_{\px,\qx} \big\|$ defines an equivalent quasi-norm in $\A$.
\end{theorem}

\begin{proof}
For $\sigma\in\R$ and $\w\in\W$ we have $(-\sigma)\w\in\mathcal{W}^\alpha_{\alpha_1-\sigma, \alpha_2-\sigma}$.

Let $\{\varphi_j\}$ be an admissible system. Consider an auxiliary system $\{\Theta_k\} \subset \cS$ such that
$$
\Theta_0(x)=1 \ \ \ \text{if} \ \ |x|\le 1, \ \ \ \ \supp \Theta_0 \subset \{x\in\Rn: |x| \le 2\}
$$
and
$$
\sum_{k=0}^\infty \Theta_k(x)= 1, \ \ \ \ \forall x\in\Rn,
$$
where
$$
\Theta_k(x)= \Theta_0(2^{-k}x) - \Theta_0(2^{-k+1}x)\,, \ \ \ \ k\in\N, \ \ x\in\Rn.
$$
Then we have
$$
\supp \Theta_k \subset \{x\in\Rn: 2^{k-1} \le |x| \le 2^{k+1}\}\,, \ \ \ \ k\in\N.
$$
With the understanding that $\Theta_{-1} \equiv 0$, by straightforward calculations we get, for $j\in\Nz$, $x\in\Rn$ and $f\in\cS'$,
\begin{eqnarray*}
 & & (2\pi)^{\frac{n}{2}} \left| \Big( \varphi_j \ast \Big[ \big(1+|\cdot|^2\big)^{\frac{\sigma}{2}} \, \hat{f}\Big]^\vee\Big)(x)\right|  =
\left| \left(  \Big[ \big(1+|\cdot|^2\big)^{\frac{\sigma}{2}} \sum_{k=-1}^{1} \Theta_{j+k}(\cdot) \Big]^\vee \ast \varphi_j \ast f\right)(x)\right| \\
& \le & \big( \varphi^\ast_j f\big)_a(x) \; \int_{\Rn} \left|\Big[ \big(1+|\cdot|^2\big)^{\frac{\sigma}{2}} \sum_{k=-1}^{1} \Theta_{j+k}(\cdot) \Big]^\vee(y)\right|\, \big(1+|2^jy|^a\big)\,dy
\end{eqnarray*}
where $a>0$ is arbitrary. We need to control the integral above, which is, up to a multiplicative constant, dominated by
$$
\sum_{k=-1}^{1} \int_{\Rn} \frac{\Big|\Big[ \big(1+|\cdot|^2\big)^{\sigma/2}  \Theta_{j+k}(\cdot) \Big]^\vee(y)\, \big(1+|2^jy|^{a+n+1}\big)\Big|}{1+|2^jy|^{n+1}}\,dy.
$$
Taking into account the pointwise inequality
\begin{equation}\label{est-derivative}
\left| D^{\gamma} \big(1+|z|^2\big)^{\sigma/2} \right| \le c_{\sigma,\gamma}\, \big(1+|z|^2\big)^{\frac{\sigma-|\gamma|}{2}}
\end{equation}
and choosing $a>0$ in the form
\begin{equation}\label{choose-a}
a=2l-n-1,
\end{equation}
for $l\in\N$ large enough, we can show that the previous sum can be estimated from above by a constant times $2^{j\sigma}$. Hence we get
\begin{equation}\label{aux7}
\left| \Big( \varphi_j \ast \Big[ \big(1+|\cdot|^2\big)^{\sigma/2} \, \hat{f}\Big]^\vee\Big)(x)\right| \lesssim 2^{j\sigma}\, \big( \varphi^\ast_j f\big)_a(x)
\end{equation}
with the implicit constant not depending on $x\in\Rn$, $j\in\Nz$, $f\in\cS'$. If we choose such $a>0$ satisfying $a>\alpha+ \tfrac{n}{\min\{p^-,q^-\}}+c_{\log}(1/q)$, from the lattice property of the mixed sequence spaces and Theorem~\ref{thm:Peetre} we get
\begin{equation}\label{lift}
\left\| I_\sigma f \,|\, B^{(-\sigma)\w}_{\px,\qx}\right\| \lesssim \left\| \big(w_j\big( \varphi^\ast_j f\big)_a\big)_j(x) \,|\, \ell_{\qx}(L_\px)\right\|
\approx \left\| f \,|\, \B\right\|.
\end{equation}
In the $F$ case we can proceed exactly in the same way. Therefore $I_\sigma$ is a continuous operator from $\A$ into $A^{(-\sigma)\w}_{\px,\qx}$. Its inverse operator, $I_{-\sigma}$, is also continuous from $A^{(-\sigma)\w}_{\px,\qx}$ into $\A$. Indeed,
$$
\left\| I_{-\sigma} g \,|\, \A\right\| = \left\| I_{-\sigma} g \,|\, A^{(\sigma)(-\sigma)\w}_{\px,\qx}\right\| \lesssim \left\| g \,|\, A^{(-\sigma)\w}_{\px,\qx}\right\|,
$$
so that $f=I_{-\sigma} g\in \A$ if $g\in A^{(-\sigma)\w}_{\px,\qx}$. Since $I_\sigma f = g$, the previous inequality yields
$$
\left\| f \,|\, \A\right\| = \left\| I_{-\sigma}(I_\sigma f) \,|\, \A\right\| \lesssim \left\| I_\sigma f \,|\, A^{(-\sigma)\w}_{\px,\qx}\right\|.
$$
Combining this with \eqref{lift} and the corresponding estimate for the $F$ space, we get the equivalence
$$
\left\| f \,|\, \A\right\| \approx \left\| I_\sigma f \,|\, A^{(-\sigma)\w}_{\px,\qx}\right\|.
$$
\end{proof}

Our result includes, in particular, the lifting property that was proved in \cite[Lemma~4.4]{DieHR09} for the spaces $F^{s(\cdot)}_{\px,\qx}$. Our proof here is completely different and it is inspired by some arguments in the proof of \cite[Lemma~3.11]{LSUYY13}. We would like to stress that the bulk of the proof above is to show the pointwise estimate \eqref{aux7}, which has nothing to do with variable exponents, and then combine it with the characterization given in Theorem~\ref{thm:Peetre}.

\section{Embeddings}\label{sec:basic-embed}

Although we aim to work with function spaces independent of the starting system $\{ \varphi_\nu\}$, the $\log$-H\"older conditions are quite strong in the sense that some results work under much weaker assumptions. This is the case of the next two statements, where the conditions assumed over there are those actually needed in the proofs. The next embeddings should then be understood to hold when the same fixed system is used for the definition of all spaces involved.

The next corollary follows immediately from the embeddings given in Lemma~\ref{lem:lpqembed}.

\begin{corollary}\label{cor:embed1}
Let $p,q,q_0,q_1\in\PP$ and $\w\in\W$.
\begin{enumerate}
\item[(i)]
If $q_0\le q_1$ (and $p^+,q_0^+,q_1^+ < \infty$ when $A=F$), then
\begin{equation*}
A^{\w}_{\px,\qzx} \hookrightarrow A^{\w}_{\px,\qumx}.
\end{equation*}
\item[(ii)]
If $p^+,q^+ < \infty$, then
\begin{equation*}
B^\w_{\px,\min\{\px,\qx\}} \hookrightarrow
\F \hookrightarrow
B^\w_{\px,\max\{\px,\qx\}}.
\end{equation*}
In particular, $B^\w_{\px,\px} = F^\w_{\px,\px}$.
\end{enumerate}
\end{corollary}

To complete the picture at the basic embeddings level, we give one more result which generalizes various results that can be found in the literature for different types of function spaces.

\begin{proposition}\label{pro:embed1}
Let $\w\in\W$, $\textbf{v}\in\mathcal{W}^{\beta}_{\beta_1,\beta_2}$, $p,q_0,q_1\in\PP$ and $\tfrac{1}{q^\ast}:= \Big(\tfrac{1}{q^-_1}-\tfrac{1}{q^+_0}\Big)_+$.\\ If $\big(\tfrac{v_j}{w_j}\big)_j\in \ell_{q^\ast}(L_\infty)$ when $A=B$ or
$\big(\tfrac{v_j}{w_j}\big)_j\in L_\infty(\ell_{q^\ast})$ and $p^+,q_0^+,q_1^+ < \infty$ when $A=F$, then
\begin{equation}\label{embed1}
A^{\w}_{\px,\qzx} \hookrightarrow A^{\textbf{v}}_{\px,\qumx}.
\end{equation}
\end{proposition}

\begin{proof}
Since
\[
A^{\w}_{\px,\qzx} \hookrightarrow A^{\w}_{\px,q_0^+} \ \ \ \ \text{and} \ \ \ \ A^{\w}_{\px,q_1^-} \hookrightarrow A^{\w}_{\px,\qumx}
\]
by Corollary~\ref{cor:embed1}, it suffices to prove the claim for constant exponents $q_0,q_1\in (0,\infty]$. For simplicity, let us write $q_0$ and $q_1$ instead of $q_0^+$ and $q_1^-$, respectively.

We consider only the case $A=F$; the other case can be proved in a similar way by using \eqref{iterated} (recall that here $q$ is constant). Let $\{\varphi_j\}$ be an admissible system.
We consider first the case $q_0\le q_1$ (so that $q^\ast=\infty$). Using the monotonicity of the discrete Lebesgue spaces, we derive
\begin{eqnarray*}
\big\|f\,|\,F^{\textbf{\textit{v}}}_{\px,q_1}\big\| & = & \left\| \big\|(v_j(\varphi_j\ast f))_j\,|\, \ell_{q_1}\big\| \,|\, L_\px \right\| \\
& \le & \Big\| \big\|(w_j^{-1}\,v_j)_j\,|\, \ell_{\infty}\big\| \, \big\|(w_j(\varphi_j\ast f))_j\,|\, \ell_{q_1}\big\| \,\big|\, L_\px \Big\| \\
& \le & \big\|(w_j^{-1}\,v_j)_j\,|\, L_{\infty}(\ell_{\infty})\big\| \, \Big\| \big\|(w_j(\varphi_j\ast f))_j\,|\, \ell_{q_0}\big\| \,\big|\, L_\px \Big\| \\
& = & \big\|(w_j^{-1}\,v_j)_j\,|\, L_{\infty}(\ell_{\infty})\big\| \, \big\|f\,|\,F^{\w}_{\px,q_0}\big\|.
\end{eqnarray*}
If $q_0> q_1$ then we have $q^\ast = \tfrac{q_0q_1}{q_0-q_1}$ when $q_0<\infty$ and $q^\ast = q_1$ if $q_0=\infty$. Let $q:=\tfrac{q_0}{q_1}>1$. The Hölder's inequality and the lattice property of the $L_\px$ space yield
\begin{eqnarray*}
\big\|f\,|\,F^{\textbf{\textit{v}}}_{\px,q_1}\big\| & = & \left\| \left( \sum_{j=0}^\infty \big|w_j^{-1}\,v_j\big|^{q_1} \big|w_j(\varphi_j\ast f)\big|^{q_1}\right)^{1/q_1} \big|\, L_\px \right\| \\
& \le & \Big\| \big\|(w_j^{-1}\,v_j)_j\,|\, \ell_{q_1q'}\big\| \, \big\|(w_j(\varphi_j\ast f))_j\,|\, \ell_{q_1q}\big\| \,\big|\, L_\px \Big\| \\
& \le & \Big\|\big\|(w_j^{-1}\,v_j)_j\,|\,\ell_{q^\ast}\big\|\,|\, L_{\infty} \Big\| \, \Big\| \big\|(w_j(\varphi_j\ast f))_j\,|\, \ell_{q_0}\big\| \,|\, L_\px \Big\| \\
& = & \big\|(w_j^{-1}\,v_j)_j\,|\, L_{\infty}(\ell_{q^\ast})\big\| \, \big\|f\,|\,F^{\w}_{\px,q_0}\big\|,
\end{eqnarray*}
which completes the proof.
\end{proof}

\begin{remark}The main arguments used in the previous proof rely on \cite[Proposition~1.1.13]{Mou01} (see also \cite[p.~273]{CaeF06}).
The case $A=B$ with constant exponents $q_0,q_1$ was studied in \cite[Theorem~2.10]{MouNS13} where the embedding $B^{\w}_{\px,q_0} \hookrightarrow B^{\textbf{\emph{v}}}_{\px,q_1}$ was proved via atomic decompositions, under the assumptions $\tfrac{v_j(x)}{w_j(x)} \lesssim 1$ for $q_0\le q_1$ and $\tfrac{v_j(x)}{w_j(x)} \lesssim 2^{-j\varepsilon}$ (for some $\epsilon>0$) for arbitrary $q_0, q_1 \in (0,\infty]$. Even in this particular situation, the hypothesis in \cite[Theorem~2.10]{MouNS13} is stronger than ours when $q_0>q_1$, since it implies $\big(\tfrac{v_j}{w_j}\big)_j\in \ell_{r}(L_\infty)$ for any $r\in (0,\infty]$. Similar comments are valid to the corresponding result for the case $A=F$ which was recently studied in \cite[Theorem~2.11]{GonMN14}.
\end{remark}

From the above proposition we immediately get the following embedding for spaces with generalized smoothness:
$$ A^{\sigma}_{\px,\qzx} \hookrightarrow A^{\tau}_{\px,\qumx}$$
for admissible sequences $\sigma$ and $\tau$ satisfying $( \sigma_j^{-1}\tau_j)_j \in \ell_{q^\ast}$. For constant exponents $p,q_0,q_1$, this result is contained in \cite[Theorem~3.7]{CaeF06} and \cite[Proposition~2.11]{CaeL13} in the cases $A=B$ and $A=F$, respectively.

For function spaces of variable smoothness, with $\w=( 2^{js_0(x)} )_j$ and $\textbf{\emph{v}}= ( 2^{js_1(x)} )_j$, the embedding \eqref{embed1} can be written as
$$ A^{s_0(\cdot)}_{\px,\qzx} \hookrightarrow A^{s_1(\cdot)}_{\px,\qumx},$$
supposing $( 2^{j(s_1(x)-s_0(x))} )_j \in \ell_{q^\ast}(L_{\infty})$ in the case $A=B$, or $( 2^{j(s_1(x)-s_0(x))} )_j \in L_{\infty}(\ell_{q^\ast})$ when $A=F$. It can be checked that both conditions are equivalent to the assumptions $(s_0-s_1)^-\ge 0$ for $q^\ast = \infty$ and $(s_0-s_1)^-> 0$ for $q^\ast <\infty$. Notice that the condition $(s_0-s_1)^-> 0$ ensures that the embedding above holds indeed for any $q_0,q_1 \in \PP$.
For $B$ spaces, this later statement was shown in \cite[Theorem~6.1(ii)]{AlmH10}. When $q_0\le q_1$ the assumption $(s_0-s_1)^-\ge 0$ can be seen as an improvement of \cite[Theorem~6.1(i)]{AlmH10}, since now the smoothness parameter may change.

\begin{theorem}\label{pro:embed-SSprime}
Let $\w\in\W$ and $p,q\in\PPlog$ (with $\max\{p^+,q^+\}<\infty$ in the $F$ case). Then
\begin{equation} \label{embed-SSprime}
\cS \hookrightarrow \A \hookrightarrow \cS'.
\end{equation}
\end{theorem}

%

\begin{proof}
By Corollary~\ref{cor:embed1} (ii) it is enough to show that $\cS \hookrightarrow \B \hookrightarrow \cS'$. We shall prove these embeddings using the lifting property. The proof of the embedding into $\cS'$ will also require, in particular, the use of the maximal characterization given by Theorem~\ref{thm:Peetre}.

\emph{Step 1}: We show that $\cS \hookrightarrow \B$. If we take $v_j(x):=2^j w_j(x)$, then $(w_jv_j^{-1})=(2^{-j})_j \in \ell_{q^-}(L_\infty)$ so that $B^{\textbf{\emph{v}}}_{\px,\infty} \hookrightarrow \B$  (by Proposition~\ref{pro:embed1}). Hence it remains to show that $\cS$ is continuously embedded into $B^\w_{\px,\infty}$ (for given $\w\in\W$ and $p\in\PPlog$), that is, one needs to show that there exists $N\in\N$ such that
\begin{equation}\label{embed-S-Binfty}
\sup_{j\in\N_0} \big\| w_j(\varphi_j\ast f)\,|\, L_\px\big\| \lesssim p_N(f),
\end{equation}
for any $f\in\cS$, where
$$\displaystyle p_N(f):=\sup_{x\in\Rn} (1+|x|)^N \sum_{|\gamma|\leq N}|D^\gamma f(x)|.$$
By straightforward calculations we find that
$$
|\varphi_j\ast f(x)| \lesssim (1+|x|)^{-N} p_N(f), \ \ \ \ x\in\Rn,
$$
(with the implicit constant depending only on $n$, $N$ and on the fixed system $\{\varphi_j\}$). From the properties of the admissible weights, we also have
$$
w_j(x) \, |\varphi_j\ast f(x)| \lesssim 2^{j\alpha_2} w_0(0)\, (1+|x|)^{\alpha-N} p_N(f), \ \ \ \ x\in\Rn, \ \ j\in\Nz, \ \ f\in\cS.
$$
Since $(1+|\cdot|)^{\alpha-N} \in L_\px$ if $(N-\alpha)p^->n$, we get \eqref{embed-S-Binfty} as long as $\alpha_2 \leq 0$ and $N>\alpha + n/p^-$. Nevertheless, the restriction on $\alpha_2$ can be overcome using the lifting property in the following way: take $\sigma>0$ large enough such that $\alpha_2 -\sigma \leq 0$. The result then follows from the composition of continuous maps:
$$
\cS \xrightarrow{\; I_\sigma\;} \cS \xrightarrow{\; id\;} B^{(-\sigma)\w}_{\px,\infty}  \xrightarrow{\;I_{-\sigma}\;} B^\w_{\px,\infty}.
$$

\emph{Step 2}: We prove the right-hand side in \eqref{embed-SSprime}. Once again (cf. Corollary~\ref{cor:embed1} (i)) it suffices to show that the space $B^\w_{\px,\infty}$ is embedded into $\cS'$. We want to prove that, for every $\psi\in\cS$, there exists a constant $c_\psi>0$ such that
\begin{equation}\label{aux}
|\langle f,\psi\rangle| \leq c_\psi \|f\,|\, B^\w_{\px,\infty}\|,  \ \ \ \ \ \forall f\in B^\w_{\px,\infty}.
\end{equation}
Let $\{\Theta_j\}$ be a system as in the proof of Theorem~\ref{thm:lift} and such that, moreover, $\Theta_0(x)$ is radially strictly decreasing for $|x|\in[1,2]$. Since for any $f\in\cS'$ the identity $\sum_{j\ge 0} \Theta_j \hat{f} = \hat{f}$ holds in $\cS'$, one has
$$
|\langle f,\psi\rangle| \leq \sum_{j\ge 0} |\langle \Theta_j^\vee\ast f, \psi\rangle|.
$$
Given $j\in\Nz$ and $m\in\Zn$, let $Q_{jm}$ denote the closed dyadic cube with sides parallel to the coordinate axes, centered at $2^{-j}m$ and with side length $2^{-j}$. Observing that the collection $\{Q_{jm}\}_{m\in\Zn}$ forms a tessellation of $\Rn$ for each $j\in\Nz$, we get, for each $N\in\N$, the estimate
$$
|\langle \Theta_j^\vee\ast f, \psi\rangle| \lesssim \sum_{m\in\Zn} 2^{jN}(1+|m|)^{-N} (1+|2^{-j}m|)^{-\alpha} \int_{Q_{jm}} |(\Theta_j^\vee\ast f)(x)|\,dx, \ \ j\in\Nz, \ f\in\cS',
$$
where we took advantage of the estimate $1+|y| \approx 1+|2^{-j}m|$ for $y\in Q_{jm}$. Since $(1+|2^{-j}m|)^{-\alpha} \lesssim w_0(y)$ for $y\in Q_{jm}$ (recall Definition~\ref{def:weights} (i)), we also get
\begin{eqnarray*}
|\langle \Theta_j^\vee\ast f, \psi\rangle| & \lesssim & \sum_{m\in\Zn} 2^{jN}(1+|m|)^{-N} \inf_{y\in Q_{jm}} w_0(y) \int_{Q_{jm}} |(\Theta_j^\vee\ast f)(x)|\,dx\\
& \lesssim & \sum_{m\in\Zn} 2^{jN}(1+|m|)^{-N} |Q_{jm}| \,\inf_{y\in Q_{jm}} \left\{w_0(y) \big( \Theta_j^{\vee\ast} f\big)_a (y)\right\}\\
& \lesssim & \sum_{m\in\Zn} 2^{j(N-n)}(1+|m|)^{-N} |Q_{jm}|^{-1/r} \left( \int_{Q_{jm}} \big[w_0(x) \big( \Theta_j^{\vee\ast} f\big)_a (x)\big]^r dx\right)^{1/r}
\end{eqnarray*}
for any positive numbers $r$ and $a$. Taking $r\leq p^-$, from Hölder's inequality with exponent $\tilde{p}(\cdot):=\px/r$ we see that the last integral power is dominated by
$$
2^{1/r} \left\| \big[w_0 \big( \Theta_j^{\vee\ast} f\big)_a\big]^r \,|\, L_{\tilde{p}(\cdot)}\right\|^{1/r} \left\| \chi_{jm} \,|\, L_{\tilde{p}^\prime(\cdot)}\right\|^{1/r}
$$
where $\chi_{jm}$ denotes the characteristic function of the cube $Q_{jm}$. Since $\tilde{p}\in\PPlog$, by \eqref{norm-charact-func-cubes} we have
$$
\left\| \chi_{jm} \,|\, L_{\tilde{p}^\prime(\cdot)}\right\| \lesssim | Q_{jm} |^{1-r/p^-}.
$$
Using this above, we estimate
\begin{equation}\label{aux1}
|\langle \Theta_j^\vee\ast f, \psi\rangle|  \lesssim 2^{j(N-n+n/p^-)} \sum_{m\in\Zn} (1+|m|)^{-N} \left\| w_0 \big( \Theta_j^{\vee\ast} f\big)_a \,|\, L_{\px}\right\|.
\end{equation}
Taking $a>\alpha+n/p^-$, by Theorem~\ref{thm:Peetre} (now with $\psi_j=\Theta^\vee_j$) we have the equivalence
$$ \left\| \big(w_j \big( \Theta_j^{\vee\ast} f\big)_a\big)_j \,|\, \ell_\infty(L_{\px})\right\| \approx \|f\,|\, B^\w_{\px,\infty}\|.$$
But $w_0(x) \leq 2^{-j\alpha_1}w_j(x)$ for every $j\in\Nz$ and $x\in\Rn$, so that
$$
\left\| w_0 \big( \Theta_j^{\vee\ast} f\big)_a \,|\, L_{\px}\right\| \lesssim 2^{-j\alpha_1} \|f\,|\, B^\w_{\px,\infty}\|, \ \ \ \ \forall j\in\Nz.
$$
Using this in \eqref{aux1} we get
$$
|\langle \Theta_j^\vee\ast f, \psi\rangle|  \lesssim 2^{j(N-n-\alpha_1+n/p^-)}\, \|f\,|\, B^\w_{\px,\infty}\|
$$
if we take $N>n$. Therefore inequality \eqref{aux} follows if we are careful in choosing $a>0$ and $n\in\N$ satisfying
$$
a>\alpha+n/p^-  \ \ \ \ \ \textrm{and} \ \ \ \ \ n<N<\alpha_1+n-n/p^-.
$$
Such a choice is always possible when $\alpha_1>1+n/p^-$. If this is not the case, then we can choose $\sigma <0$ small enough such that
$\alpha_1-\sigma>1+n/p^-$. Then the result above can be applied to the space $B^{(-\sigma)\w}_{\px,\infty}$ and from the lifting property we get the claim for the general case (similarly to what was done in Step 1).
\end{proof}

\begin{remark}
Step 2 of the proof above was inspired by the proof of \cite[Theorem~3.14]{LSUYY13}. There is an alternative proof of embeddings \eqref{embed-SSprime} which takes advantage of Sobolev type embeddings for Besov spaces  with constant $q$ (cf. \cite[Theorem~2.11]{MouNS13}, \cite[Theorem~2.13]{GonMN14}). We refer to our paper \cite{AlmC15b} for the topic of Sobolev type embeddings in full generality.
\end{remark}

%
%


\section{Fourier multipliers}\label{sec:Fmultipliers}

Fourier multipliers constitute an important tool from the point of view of applications to partial differential equations. By a Fourier multiplier for $\A$ we mean a function (or a tempered distribution) $m$ for which
\begin{equation*}
\big\|\big(m\,\hat{f}\big)^\vee\,|\A\big\| \lesssim \|f\,|\A\|\,, \ \ \ \ \ f\in\A.
\end{equation*}

If $m$ is a $C^\infty$ function  such that it and all of its derivatives are at most of polynomial growth on $\Rn$ (i.e., they are essentially dominated by powers of the type  $(1+|\cdot|)^M$ with $M \geq 0$), then $\big(m\,\hat{f}\big)^\vee$ makes sense for any $f\in\cS'$.

Our first result concerning Fourier multipliers (involving that class of functions) can be proved following the proof of Theorem~\ref{thm:lift}, replacing $(1+|z|^2)^{\sigma/2}$ by $m(z)$ and using the estimate
\begin{equation*}
\left| D^{\gamma} m(z) \right| \le  \Big( \sup_{x\in\Rn} (1+|x|^2)^{\frac{|\gamma|}{2}} \,|D^{\gamma} m(x)|\Big)\,(1+|z|^2)^{-\frac{|\gamma|}{2}}
\end{equation*}
instead of \eqref{est-derivative}. Hence the dependence on $\sigma$ over there should be replaced along the proof by
$$
\|m\|_{2l}:= \sup_{|\gamma| \le 2l} \sup_{x\in\Rn} (1+|x|^2)^{\frac{|\gamma|}{2}} \, |D^\gamma m(x)|\,,
$$
where $l\in\N$ satisfies now  $a+n+\varepsilon=2l$ for some $\varepsilon >0$ (playing the role of the $1$ in \eqref{choose-a}) and $a>\alpha+ \tfrac{n}{\min\{p^-,q^-\}}+c_{\log}(1/q)$. We leave the details to the reader.

\begin{theorem}\label{thm:Fmultiplier1}
Let $\w\in\W$ and $p,q\in\PPlog$ ($\max\{p^+,q^+\}<\infty$ in the $F$ case). Let $l\in\N$ be such that
\begin{equation*}
2l > \alpha+\frac{n}{p^-} + n + c_{\log}(1/q)\ \ \textrm{(in the $B$ case)} \ \ \ \text{or} \ \ \  2l > \alpha+ \tfrac{n}{\min\{p^-,q^-\}}+n \ \ \textrm{(in the $F$ case)}.
\end{equation*}
Then there exists $c>0$ such that
\begin{equation*}
\big\|\big(m\,\hat{f}\big)^\vee\,|\A\big\| \leq c\, \|m\|_{2l}\, \|f\,|\A\|
\end{equation*}
for all $m$ as considered before and all  $f\in\A$.
\end{theorem}

The particular case $w_j(x)=2^{js(x)}$, corresponding to spaces with variable smoothness, was recently studied in \cite{Noi14}, but only bounded exponents were considered there (even in the $B$ space). The proofs over there follow the arguments from \cite[Theorem~2.3.7]{Tri83} where the constant exponent case is treated.

Recalling the notation from Theorem~\ref{thm:lift} we get the following corollary:

\begin{corollary}
Let $\w\in\W$ and $p,q\in\PPlog$ (with $\max\{p^+,q^+\}<\infty$ in the $F$ case). For any $\gamma \in\mathbb{N}_0^n$ the differentiation operator $D^\gamma$ is continuous from $A^{(|\gamma|)\w}_{\px,\qx}$ into $\A$.
\end{corollary}

\begin{proof}
Let $m(x):=x^\gamma (1+|x|^2)^{-\frac{\kappa}{2}}$ with $\kappa\in\Nz$. We can check that $m$ and all of its derivatives are $C^\infty$ functions of at most polynomial growth and, additionally, that if $\kappa \ge |\gamma|$ then $\|m\|_{2l} <\infty$ for all $l\in\N$. In particular we can take $l$ large enough in order to apply Theorem~\ref{thm:Fmultiplier1} and conclude that the above $m$ (with that choice of $\kappa$) is a Fourier multiplier for every space $\A$.\\
Consider now $\kappa = |\gamma|$ and let $f\in A^{(|\gamma|)\w}_{\px,\qx}$ (recall that $(|\gamma|)\w:=(2^{j|\gamma|}w_j)_j$). By Theorem~\ref{thm:lift} we know that $I_\kappa f \in A^{(|\gamma|-\kappa)\w}_{\px,\qx} = \A$ and that $\big\| I_\kappa f \,|\, \A \big\| \approx \big\| f\,|\, A^{(|\gamma|)\w}_{\px,\qx} \big\|$. Thus we get
\begin{eqnarray*}
\big\| D^\gamma f \,|\, \A \big\| & = &\big\| (x^\gamma \, \hat{f}(\cdot))^\vee \,|\, \A \big\| = \big\| \big( m(\cdot)\widehat{I_\kappa f}\big)^\vee \,|\, \A \big\| \\
& \leq & c\, \|m\|_{2l}\, \|I_\kappa f\,|\A\| \approx \big\| f \,|\, A^{(|\gamma|)\w}_{\px,\qx} \big\|,
\end{eqnarray*}
which proves the claim.
\end{proof}

\begin{remark}
As in \cite[pp.~59-60]{Tri83}, from which the argument above is taken, fixing $\kappa\in\N$, for any $\gamma\in\N_0^n$ such that $|\gamma|\le \kappa$, from Proposition~\ref{pro:embed1} and the result above we see that
$$
\left\|  D^\gamma f \,|\, A^{(-k)\w}_{\px,\qx} \right\| \lesssim \left\|  D^\gamma f \,|\, A^{(-|\gamma|)\w}_{\px,\qx} \right\| \lesssim \left\| f \,|\, \A \right\|
$$
and consequently
$$
\sum_{|\gamma|\le \kappa} \left\|  D^\gamma f \,|\, A^{(-\kappa)\w}_{\px,\qx} \right\| \lesssim \left\| f \,|\, \A \right\|.
$$
\end{remark}

It can be shown that $\cS$ is dense in $\A$ when $p,q$ are bounded (and $p,q \in \PPlog$). This statement was formulated in \cite[Theorem~2.13]{GonMN14} for the $F$ case and in \cite[Theorem~2.11]{MouNS13} for the spaces $\B$ for constant $q$ only. The proofs mentioned over there rely on classical arguments from \cite{Tri83} (cf. proof of Theorem~2.3.3). However, we can show that the same result  also holds for variable exponents $q$ in the $B$ case (when $p^+,q^+<\infty$ and $p,q \in \PPlog$). Note that this corresponds to the usual restrictions $p,q<\infty$ already known for constant exponents.

In the paper \cite{AlmC15b} we give a complete proof of the denseness of $\cS$ in $\A$  by taking advantage of refined results on atomic decompositions. Thus, at least in the cases $p^+,q^+<\infty$, we can then define Fourier multipliers by completion.
\begin{definition}
Let $\w\in\W$ and $p,q\in\PPlog$ with $\max\{p^+,q^+\}<\infty$. Then $m\in\cS'$ is said to be a Fourier multiplier for $\A$ if
\begin{equation*}
\big\|\big(m\,\hat{f}\big)^\vee\,|\A\big\| \lesssim \|f\,|\A\|
\end{equation*}
holds for all $f\in\cS$.
\end{definition}

Notice that this definition is coherent with the interpretation made before when $m\in\cS'$ is, in particular, a $C^\infty$ function with at most polynomial growth as well as all of its derivatives.

Before stating the next result, we observe that there are functions $\lambda_j\in\cS$, $j\in\Nz$, such that
$$
\lambda_0(x)=1 \ \ \text{if} \ \ |x|\le 2 \ \ \ \ \text{and} \ \ \ \ \lambda_0(x)=0 \ \ \text{if} \ \ |x|\ge 4,
$$
$$
\lambda(x):= \lambda_0(x) - \lambda_0(8x), \ \ \ \ \lambda_j(x)=\lambda(2^{-j}x), \ \ \ \ j\in\N, \ \ x\in\Rn,
$$
so that
$$
\lambda_j(x)=1 \ \ \text{if} \ \ 2^{j-1} \le |x|\le 2^{j+1} \ \ \ \ \text{and} \ \ \ \ \lambda_j(x)=0 \ \ \text{if} \ \ |x|\le 2^{j-2} \ \ \text{or} \ \ |x|\ge 2^{j+2}
$$
for $j\in\N$.

Below we use the notation
$$
\|m\,|\,h^{\kappa}_2\|:= \|\lambda_0\,m\,|\,H^{\kappa}_2\| + \sup_{j\in\N} \|\lambda(\cdot)\,m(2^j\cdot)\,|\,H^{\kappa}_2\|,
$$
where
$$
H^{\kappa}_2 =\big\{f\in \cS': \, \|f\,|\,H^{\kappa}_2\| = \big\| (1+|\cdot|^2)^{\kappa/2} \hat{f} \,|\,L_2\big\| < \infty \big\}
$$
stands for the classical Bessel potential space, $\kappa\in\N$. As usual, we write $m\in h^{\kappa}_2$ meaning that $\|m\,|\,h^{\kappa}_2\|<\infty$.

\begin{theorem}
Let $\w\in\W$ and $p,q\in\PPlog$ with $\max\{p^+,q^+\}<\infty$. Let
\begin{equation*}
\kappa > \alpha+\frac{n}{p^-} + \frac{n}{2} +c_{\log}(1/q)\ \ \textrm{(in the $B$ case)} \ \ \ \text{or} \ \ \  \kappa > \alpha+\frac{n}{\min\{p^-,q^-\}} + \frac{n}{2} \ \ \textrm{(in the $F$ case)}.
\end{equation*}
Then there exists $c>0$ such that
\begin{equation*}
\big\|\big(m\,\hat{f}\big)^\vee\,|\A\big\| \leq c\,\|m\,|\,h^{\kappa}_2\|\,\|f\,|\A\|
\end{equation*}
for all $m\in h^{\kappa}_2$ and $f\in\A$.
\end{theorem}

\begin{proof}
It is enough to prove the claim for all $f\in\cS$ (and all $m\in h^{\kappa}_2$). We try to adapt the proof of Theorem~\ref{thm:lift} to our present context, using now the functions $\lambda_j$ instead of $\Theta_j$ over there.  Let $\{\varphi_j\}$ be an admissible system. Noting that $\lambda_j \equiv 1$ on support of $\hat{\varphi}_j$ and that $(\lambda_j(\cdot)\,m)^\vee$ is a entire analytic function of at most polynomial growth (by the Paley-Wiener-Schwartz theorem), for each $j\in\Nz$ and $x\in\Rn$, we have
$$
(2\pi)^{\frac{n}{2}}\, \left| \Big( \varphi_j \ast \big(m\,\hat{f}\big)^\vee \Big)(x)\right|  \le  \big( \varphi^\ast_j f\big)_a(x) \; \int_{\Rn} \big(1+|2^jy|^a\big)\, \left|[\lambda_j(\cdot)\,m]^\vee(y)\right|\, dy
$$
with $a>0$. We need to control appropriately the integral above. After the change of variables given by $2^jy=z$, an application of Schwarz's inequality yields
\begin{eqnarray*}
& & \int_{\Rn} (1+|z|^a)\, \left|[\lambda_j(\cdot)\,m]^\vee(2^{-j}z)\right|\,2^{-jn}\,dz \, \le \quad\quad\quad \\
& \le & c_1 \int_{\Rn} (1+|z|^2)^{\frac{1}{2} (a+\frac{n+\varepsilon}{2})}\, \left|[\Lambda_j(\cdot)\,m(2^j \cdot)]^{\wedge}(-z)\right|\,(1+|z|)^{-\frac{n+\varepsilon}{2}}\,dz \\
& \le & c_2 \left(\int_{\Rn} (1+|z|^2)^{a+\frac{n+\varepsilon}{2}}\, \left|[\Lambda_j(\cdot)\,m(2^j \cdot)]^{\wedge}(z)\right|^2\,dz\right)^{1/2}
\end{eqnarray*}
with $c_2>0$ depending only on $a$, $n$ and $\varepsilon$ ($\varepsilon>0$ at our disposal), where we used $\Lambda_j:=\lambda_j(2^j\cdot)$, $j\in\Nz$.
By Plancherel's theorem we further estimate the right-hand side of the inequality above by
$$
c_2 \left\| \Lambda_j(\cdot)\,m(2^j \cdot) \,|\, H_2^{a+\frac{n+\varepsilon}{2}} \right\| = c_2 \left\| \Lambda_j(\cdot)\,m(2^j \cdot) \,|\, H_2^{\kappa} \right\|
\le c_2 \left\| m \,|\, h_2^{\kappa} \right\|,
$$
by choosing $\varepsilon>0$ such that
$$
\kappa=\alpha +\frac{n}{p^-} +\frac{\varepsilon}{2} + \frac{n+\varepsilon}{2} +c_{\log}(1/q)\ \ \ \text{and} \ \ \ a=\alpha+ \frac{n}{p^-} +\frac{\varepsilon}{2} +c_{\log}(1/q)
$$
in the $B$ case, or
$$
\kappa=\alpha + \frac{n}{\min\{p^-,q^-\}} +\frac{\varepsilon}{2} + \frac{n+\varepsilon}{2} \ \ \ \text{and} \ \ \ a=\alpha + \frac{n}{\min\{p^-,q^-\}} +\frac{\varepsilon}{2}
$$
in the $F$ case, which is clearly possible by the assumption on $\kappa$. Putting everything together, we obtain the estimate
$$\left| \Big( \varphi_j \ast \big(m\,\hat{f}\big)^\vee \Big)(x)\right| \lesssim \left\| m \,|\, h_2^{\kappa} \right\| \, \big( \varphi^\ast_j f\big)_a(x)\,, \ \ \ j\in\Nz, \ \ x\in\Rn,
$$
Therefore, the desired inequality follows from this, the lattice property of the mixed sequence spaces and Theorem~\ref{thm:Peetre}.
\end{proof}

Similar results for the particular case $w_j(x)=2^{js(x)}$ are given in \cite{Noi14}, although our assumptions on $\kappa$ are better. Corresponding statements for classical spaces $A^s_{p,q}$ (then with even better assumptions on $\kappa$) can be seen in \cite[p.~117]{Tri83}, \cite[Proposition~1.19]{Tri13}, though with different restrictions.

Since the Triebel-Lizorkin scale includes Bessel potential spaces, in particular we have the following Fourier multipliers result:

\begin{corollary}
Let $s\in[0,\infty)$ and $p\in\PPlog$ with $1<p^-\leq p^+<\infty$. If $\kappa > \frac{n}{\min\{p^-,2\}} + \frac{n}{2}$, then there exists $c>0$ such that
\begin{equation*}
\big\|\big(m\,\hat{f}\big)^\vee\,|\,\mathcal{L}^{s,\px}\big\| \leq c\,\|m\,|\,h^{\kappa}_2\|\,\|f\,|\,\mathcal{L}^{s,\px}\|
\end{equation*}
for all $m\in h^{\kappa}_2$ and $f\in\mathcal{L}^{s,\px}$.\\
In particular, for any $k\in\Nz$ we have
\begin{equation*}
\big\|\big(m\,\hat{f}\big)^\vee\,|\,W^{k,\px}\big\| \leq c\,\|m\,|\,h^{\kappa}_2\|\,\|f\,|\,W^{k,\px}\|
\end{equation*}
for all $m\in h^{\kappa}_2$ and $f\in W^{k,\px}$.
\end{corollary}

\end{document}